\documentclass[12pt]{amsart}

\usepackage{tikz}
\usetikzlibrary{calc,matrix, arrows, babel, patterns}
\usepackage{tikz-cd}

\usepackage[utf8]{inputenc}
\usepackage[T1]{fontenc}

\usepackage{amsmath,amsthm,amsfonts,amssymb,latexsym}

\usepackage{hyperref}

\usepackage{tabto}

\usepackage{cleveref}

\usepackage{color, xcolor}

\usepackage{enumerate, enumitem, verbatim, graphicx}

\usepackage{mathrsfs, thmtools}

\usepackage{setspace}

\theoremstyle{theorem}
\newtheorem*{thmA}{Theorem A}
\newtheorem*{thmB}{Theorem B}

\newtheorem{teo}{Theorem}

\newtheorem{lem}[teo]{Lemma}
\newtheorem{cor}[teo]{Corollary}

\theoremstyle{definition}

\newtheorem*{ack}{Acknowledgments}

\def\cent#1#2{{\bf C}_{#1}(#2)}

\def\cent#1#2{{\bf C}_{#1}(#2)}
\newcommand{\fA} {\mathfrak A}

\newcommand{\fS} {\mathfrak S}

\newcommand{\FF}{{\mathbb{F}}}

\newcommand{\bC} {\mathbf C}
\newcommand{\bE} {\mathbf E}
\newcommand{\bF} {\mathbf F}

\newcommand{\bO}{{\mathbf O}}
\newcommand{\bN} {\mathbf N}

\newcommand{\bZ} {\mathbf Z}

\newcommand{\Aut}{{\operatorname{Aut}}}

\newcommand{\Irr}{{\operatorname{Irr}}}

\newcommand{\Stab}{{\operatorname{Stab}}}
\newcommand{\Syl}{{\operatorname{Syl}}}

\newcommand{\GL}{\operatorname{GL}}
\newcommand{\AGL}{\operatorname{AGL}}
\newcommand{\GaL}{\operatorname{\Gamma L}}
\newcommand{\AGaL}{\operatorname{A\Gamma L}}
\newcommand{\SL}{\operatorname{SL}}

\newcommand{\PSL}{\operatorname{PSL}}
\newcommand{\PSU}{\operatorname{PSU}}

\begin{document}

\title[Height zero Galois]{Height zero characters and Galois automorphisms}


\author{Alexander Moretó, Noelia Rizo and Gabriel A. L. Souza}
\address{Departament de Matemàtiques, Universitat de València, 46100 Burjassot, València, Spain}
\email{alexander.moreto@uv.es, noelia.rizo@uv.es, gabriel.area@uv.es}


\thanks{The authors are supported by Ministerio de Ciencia e Innovación (project PID2022-137612NB-I00 funded by MCIN/AEI/10.13039/501100011033 and ``ERDF A way of making Europe''). The third author is supported by grant PREP2022-000021 tied to said project.}



\date{\today}


\begin{abstract}
	Let $G$ be a finite group and let $p$ be a prime. In this paper, we prove a strengthened version of Brauer's height zero conjecture for the principal $p$-block of $G$ that takes the action of a certain group of Galois automorphisms into account. This answers a conjecture recently proposed by Malle, Moretó, Rizo and Schaeffer Fry. We then use this to obtain a structural result which can be seen as a Galois version of the It\^o--Michler theorem.
\end{abstract}

\maketitle


\setstretch{1.15}

	\section{Introduction}
	
	Many local-global conjectures in the representation theory of finite groups have been generalized to take Galois automorphisms into account. For example, the McKay conjecture, recently proven by Cabanes and Späth in \cite{McKay}, received such a generalization in \cite{Galois-McKay}.
	
	Another instance of this phenomenon is the Brauer Height Zero Conjecture, which was also recently proven in \cite{BHZ}. The principal block case of this conjecture asserts that if $p$ is a prime and $G$ is a finite group then all irreducible characters in the principal $p$-block of $G$ have $p'$-degree if and only if $G$ has an abelian Sylow $p$-subgroup. 
	
	 In \cite{MN}, the authors considered the action of a specific automorphism $\sigma$ on the characters of the principal $2$-block of a finite group. In that context, they showed that it sufficed to consider the set of $\sigma$-invariant irreducible characters in the statement of Brauer's Height Zero Conjecture, for $p = 2$.
	 
	 Let $p$ be a prime number and let $\mathbb{Q}^{\operatorname{ab}}$ be the maximal abelian extension of $\mathbb{Q}$. Then, we write $\mathcal{J}$ to denote the set of elements of order $p$ of the group $\operatorname{Gal}(\mathbb{Q}^{\operatorname{ab}}/\mathbb{Q})$ that fix all $p$-power roots of unity. In \cite{MMRSF}, it was shown that, in fact, it sufficed to consider the set of $\mathcal{J}$-invariant characters in the statement of Brauer's Height Zero Conjecture, provided the group $G$ does not contain a certain set of composition factors for specific primes $p$.
	
	In this paper, we remove this restriction on the composition factors and prove the theorem below, which is \cite[Conjecture 1]{MMRSF}. In the following, $B_0(G)$ denotes the principal $p$-block of $G$ and $\Irr_{\mathcal{J}}(G)$, the set of irreducible complex characters of $G$ that are invariant under $\mathcal{J}$.
	
	\begin{thmA}
		Let $G$ be a finite group, let $p$ be a prime number and let $P\in{\rm Syl}_p(G)$. Suppose that the elements in ${\rm Irr}_{\mathcal{J}}(B_0(G))$ are all of $p'$-degree. Then $P$ is abelian.
	\end{thmA}
	
	In order to do so, we use many results of \cite{MMRSF}, especially those relying on the Classification of Finite Simple Groups (CFSG). However we also employ substantially different techniques from those of \cite{MMRSF}. In fact, the approach in \cite{MMRSF}, as that in \cite{MN}, went through a Galois version of the It\^o--Michler theorem, which was known to have  exceptions with our set of Galois automorphisms $\mathcal{J}$. 
	
Our approach now relies more on arguments from \cite{MinimalHeights}. In particular, we use deep classification theorems on the structure of primitive permutation groups. As a consequence of Theorem A, we are also able to prove the following, which is probably the closest one can get to a Galois version of It\^o--Michler.
	
	\begin{thmB}
		Let $G$ be a finite group, with $p$ a prime and suppose that $\Irr_{\mathcal{J}}(G) \subseteq \Irr_{p'}(G)$.
		Then, $\bO^{p'}(G) = \bO_p(G) \times K$, where $K \unlhd G$, $\bO_{p'}(K)$ is solvable and $K/\bO_{p'}(K)$ is a direct product of non-abelian simple groups of order divisible by $p$, whose $\mathcal{J}$-invariant irreducible characters are all of $p'$-degree.
	\end{thmB}
	
	We remark that Theorem B  generalizes the main result of \cite{Grit}. Unfortunately, as we discuss in Section 4, it appears that the proof in \cite{Grit} is not correct.
	
	\section{Preliminary results}
	
	In this section, we outline some of the main tools that will be used to prove Theorems A and B. To do so, we begin by substituting the group $\mathcal{J}$, which is an infinite group, for a finite one that has the same action on the characters of a finite group $G$.
	
	Given a natural number $n$, if we write $\mathbb{Q}_n = \mathbb{Q}(\xi)$, where $\xi$ is a primitive $n$-root of unity, and $n_p$ for the biggest power of $p$ that divides $n$, then we define $\Omega(n)$ as the set of elements of order $p$ of $\operatorname{Gal}(\mathbb{Q}_n/\mathbb{Q}_{n_p})$. The interplay between $\mathcal{J}$ and $\Omega(n)$ for the context of finite groups is given by the following result.
	
	\begin{lem}
		Let $G$ be a finite group and let $\chi \in \Irr(G)$. Then, $\chi$ is $\Omega(|G|)$-invariant if and only if it is $\mathcal{J}$-invariant.
	\end{lem}
	
	\begin{proof}
		This is \cite[Lemma 2.1]{MMRSF}.
	\end{proof}
	
	Even though $\Omega$ depends on $n$, for groups we can usually get around this by using the lemma below. As such, we will work with the sets $\Omega(|G|)$ instead of $\mathcal{J}$ for the remainder of the paper. We will write $\Omega$ for $\Omega(|G|)$, if the finite group $G$ is clear from context, and write $\Irr_{\Omega}(G)$ for the characters of $G$ that are $\Omega$-invariant.
	
	\begin{lem}
		Let $G$ be a finite group, let $\chi \in \Irr(G)$ and let $m$ be a multiple of $|G|$. Then, $\chi$ is $\Omega(|G|)$-invariant if and only if it is $\Omega(m)$-invariant.
	\end{lem}
	
	\begin{proof}
		This is \cite[Corollary 2.2]{MMRSF}.
	\end{proof}
	
	The three results below, all from \cite{GLPST} and relying on the CFSG, are the deep classification theorems mentioned in the Introduction. A subgroup $H \leq \GL_n(V)$ is called \emph{$p$-exceptional} if $p$ divides $|H|$, but none of its orbits on the vectors of $V$ has size divisible by $p$. Analogously, a subgroup $H \leq \fS_{\Lambda}$ is called \emph{$p$-concealed} if none of its orbits on the subsets of $\Lambda$ has size divisible by $p$.
	
	\begin{teo}   \label{t1}
		Let $G$ be an irreducible $p$-exceptional subgroup of $\GL_n(p)=\GL(V)$
		and suppose that $G$ acts primitively on $V$. Then one of the following holds:
		\begin{enumerate}
			\item $G$ is transitive on $V\setminus\{0\}$;
			\item $G\leq\GaL_1(p^n)$;
			\item $G$ is one of the following:
			\begin{enumerate}
				\item[(i)] $G=\fA_c, \fS_c$ where $c=2^r-2$ or $2^r-1$, with $V$ the deleted
				permutation module over $\FF_2$, of dimension $c-2$ or $c-1$
				respectively;
				\item[(ii)] $\SL_2(5)\trianglelefteq G< \GaL_2(9)<\GL_4(3)$,
				orbit sizes $1,40,40$;
				\item[(iii)] $\PSL_2(11)\trianglelefteq G<\GL_5(3)$, orbit sizes $1,22,110,110$;
				\item[(iv)] $M_{11}\trianglelefteq G<\GL_5(3)$, orbit sizes $1,22,220$;
				\item[(v)] $M_{23}=G<\GL_{11}(2)$, orbit sizes $1,23,253,1771$.
			\end{enumerate}
		\end{enumerate}
	\end{teo}
	
	\begin{proof}
		This is Theorem 1 of \cite{GLPST}.
	\end{proof}
	
	\begin{teo}   \label{t3}
		Suppose $G\leq \GL_n(p)=\GL(V)$ is irreducible and $p$-exceptional with
		$G=\bO^{p'}(G)$. If $V=V_1\oplus\cdots\oplus V_n$ ($n>1$) is an
		imprimitivity decomposition for $G$, then $G_{V_1}$ is transitive on
		$V_1\setminus\{0\}$ and $G$ induces a primitive $p$-concealed subgroup of
		$\fS_\Lambda$, where $\Lambda=\{V_1,\dots, V_n\}$.
	\end{teo}
	
	\begin{proof}
		This is \cite[Thm~3]{GLPST}.
	\end{proof}
	
	\begin{teo}
		
		\label{t2}
		Let $H$ be a primitive subgroup of $\fS_n$ of order divisible by a prime $p$.
		Then $H$ is $p$-concealed if and only if one of the following holds:
		\begin{enumerate}
			\item $\fA_n\trianglelefteq H\leq \fS_n$, and $n=ap^s-1$ with $s\geq 1$,
			$a\leq p-1$ and $(a,s)\neq(1,1)$; also $H\neq\fA_3$ if $(n,p)=(3,2)$;
			\item $(n,p)=(8,3)$, and $H=\AGL_3(2)=2^3:\SL_3(2)$ or
			$H=\AGaL_1(8)=2^3:7:3$;
			\item $(n,p)=(5,2)$ and $H=D_{10}$. 
		\end{enumerate}
	\end{teo}
	
	\begin{proof}
		This is \cite[Thm~2]{GLPST}.
	\end{proof}
	
	One consequence of this latter classification which will be of particular use to us is the following variation of \cite[Lemma 3.4]{MinimalHeights}.
		
	\begin{lem}   \label{pconomega}
		Let $p$ be prime number and let $H$ be a primitive subgroup of $\fS_n$ of order
		divisible by $p$. If $H$ is $p$-concealed, then there exists $\chi\in\Irr_\Omega(H)$ such that
		$\chi(1)_p>1$.
	\end{lem}
	\begin{proof}
		We go through the cases of Theorem \ref{t2}. Suppose first we are in case (1). If $n\leq 4$, we get that $p=2$, as $(a, s) \neq (1, 1)$, and $H \neq \fA_3$. Thus, the only possible case is $H = \fS_3$, where we can find a rational character of degree divisible by $2$. Hence we may assume that $n\geq 5$. If $H=\fS_n$, all irreducible characters of $H$ are rational and we are done by the Itô--Michler theorem, since $\fS_n$ does not have normal Sylow $p$-subgroups for any prime $p$. Finally, if $H = \fA_n$, by \cite[Theorem 2.8]{MMRSF}, a counterexample could only happen if $p \neq 2$ and $n \in \{5, 6, 8\}$. It is straightforward to check (for example, with GAP \cite{gap}) that the result holds in all of these cases.
				
		Now suppose we are in case (2).. Then, we may compute the character tables of the two groups explicitly using GAP and we see that $\AGL_3(2)$ possesses a rational (hence $\Omega$-invariant) character of degree divisible by $3$ and that $\AGaL_1(8)$ has two characters of degree $3$; since $\Omega$ is a $3$-group in this case, both must be $\Omega$-invariant. Finally, for case (3), $p = 2$ and $\Omega$ is the set of $2$-elements of $\operatorname{Gal}(\mathbb{Q}_{10}/\mathbb{Q})$. There is exactly one such element besides the identity, which is complex conjugation. Since all characters of $D_{10}$ are real, its characters of degree $2$ are $\Omega$-invariant.
	\end{proof}
	
	Next, we state one property of simple groups which will be of fundamental importance to us.
	
	\begin{lem} \label{simpleGroupCondition}
		Let $S$ be a non-abelian finite simple group and let $p$ be an odd prime dividing $|S|$. Then, there exists $1_S \neq \alpha \in \Irr_\Omega(B_0(S))$ of $p'$-degree. 
	\end{lem}
	
	\begin{proof}
		Suppose first that $p = 3$. By \cite[Corollary 1.6]{Landrock}, $3 \mid |\Irr_{p'}(B_0(S))|$. Since $\Omega$ is a $3$-group, we have $|\Irr_{p', \Omega}(B_0(S))| \equiv |\Irr_{p'}(B_0(S))| \pmod{3}$, meaning (since $1_S$ is of $3'$-degree and $\Omega$-invariant) there are at least two non-trivial characters in $\Irr_{p', \Omega}(B_0(S))$. We may thus assume $p > 3$.
		
		For alternating groups, there is nothing to prove, since every irreducible character is $\Omega$-invariant (as $\Omega$ is a $p$-group, $p > 2$ and characters of $\fA_n$ have at most $2$ Galois conjugates).
		
		For sporadic groups (and the Tits simple group $^2F_4(2)'$), we may check with GAP \cite{gap} that, if $p \mid |S|$, then there exists a rational character of $p'$-degree in the principal $p$-block of $S$. Finally, if $S$ is a simple group of Lie type, the result follows by Theorems 5.4 and 5.5 of \cite{NTV}.
	\end{proof}
	
	We also need a couple of results that allow us to deal with normal subgroups and certain characters thereof. 
	
	\begin{lem}
		\label{3.2BHZG}
		Suppose that $N$ is a normal subgroup of $G$, with $G/N$ a $p'$-group. Let $\Omega = \Omega(|G|)$. If
		$\theta \in \Irr_{\Omega}(N)$, then there exists $\chi \in \Irr_{\Omega}(G)$ over $\theta$. Furthermore, if $\theta \in \Irr_{\Omega}(B_0(N))$, then there exists $\chi \in \Irr_{\Omega}(B_0(G))$ over $\theta$.
	\end{lem}
	\begin{proof}
		This is \cite[Lemma 3.2]{MMRSF}.
	\end{proof}
	
	\begin{lem} \label{liftingTheta}
		Let $N\lhd G$ and let $\theta\in{\rm Irr}_{p',\Omega}(B_0(N))$. Suppose that $\theta$ is $G$-invariant and that $p$ does not divide $o(\theta)$. Then there exists $\chi\in{\rm Irr}_{p',\Omega}(B_0(G)|\theta).$
	\end{lem}
	\begin{proof}
		Let $R\in{\rm Syl}_p(G)$. Since $\theta$ is $G$-invariant and $(\theta(1)o(\theta),|RN/N|)=1$, $\theta$ has a canonical extension $\psi\in{\rm Irr}_{p',\Omega}(RN)$ and since $RN/N$ is a $p$-group, $\psi$ lies in $B_0(RN)$. Now by the Alperin--Dade correspondence (see \cite[Lemma 3.1]{MMRSF}, for instance), we have that $\psi$ has an extension $\hat\psi\in{\rm Irr}_{p',\Omega}(B_0(RN\cent G R))$.
		
		Now $\Delta=\hat\psi^G$ is $\Omega$-invariant and has degree coprime to $p$. We consider $\Delta_{B_0(G)}$, the sum of the irreducible constituents of $\Delta$ that belong to $B_0(G)$, including multiplicities. Now write $X=RN\cent G R$. By the argument in \cite[page 213]{NavarroBlocks}, we have that $B_0(X)^G$ is defined. By Brauer's Third Main Theorem, we have that $B_0(X)^G=B_0(G)$. Now by \cite[Corollary 6.4]{NavarroBlocks}, we have that
		$$\Delta_{B_0(G)}(1)_p=((\hat{\psi}^G)_{B_0(G)})(1)_p=\hat{\psi}^G(1)_p=1.$$
		
		Notice that $\Delta_{B_0(G)}$ is $\Omega$-invariant. Indeed, if $\chi\in{\rm Irr}(\Delta)$ lies in the principal block, then, since $\Delta$ is $\Omega$-invariant, if $\sigma\in \Omega$, we have that $\chi^\sigma\in{\rm Irr}(\Delta^\sigma)={\rm Irr}(\Delta)$ and $\chi^\sigma$ lies in the principal block, hence $\chi^\sigma\in{\rm Irr}(\Delta_{B_0(G)})$. Now we apply Lemma 2.1 (ii) of \cite{NT} to $\Delta_{B_0(G)}$ to conclude that there exists $\chi\in{\rm Irr}_{p'}(\Delta_{B_0(G)})$, $\Omega$-invariant. Then $\chi\in{\rm Irr}_{p',\Omega}(B_0(G)|\theta)$, as desired.
	\end{proof}
	
	To finish off this section, the following allows us to deal with the case of groups with a semisimple minimal normal subgroup which is not simple.
	
	\begin{lem}
		\label{SemisimpleCase}
		Let $G$ be a finite group and let $p$ be a prime such that $G = \bO^{p'}(G)$. Assume $G$ contains a minimal normal subgroup $N$ which is isomorphic to $S_1 \times \cdots \times S_k$, for isomorphic non-abelian finite simple groups $S_i$, with $k > 1$. Then, there exists $\chi \in \Irr_{\Omega}(G)$ such that $\chi(1)_p > 1$. Furthermore, if $p$ is odd and divides $|S_1|$, such a $\chi$ can be taken in the principal block.
	\end{lem}
	
	\begin{proof}
		We will tackle both the general case and the case where $p$ is odd and divides $|S_1|$ at the same time, making the necessary modifications as we go. Let $H = \bigcap_{i=1}^k \bN_{G}(S_i)$ and let $\overline{G} = G/H$, which acts transitively on the set $\Lambda = \{S_1, \ldots, S_k\}$ by conjugation. Notice how $H \neq G$, as $k > 1$.
		
		Suppose, first, that this action is not $p$-concealed, meaning there exists a subset $\Gamma$ of $\Lambda$ such that $p$ divides $|\overline{G}:\Stab_{\overline{G}}(\Gamma)|$. Write $S = S_1$. In the general case, let $1_S \neq \phi \in \Irr(S)$ be a rational character that extends to $\Aut(S)$ (such a character exists by \cite[Lemma 4.1]{HSFTVV}). If $p$ is odd and divides $|S|$, let instead $1_S \neq \phi \in \Irr_{p', \Omega}(B_0(S))$ be a character as in Lemma \ref{simpleGroupCondition}. In any case, let $\theta$ be the character of $N$ obtained by placing $\phi$ in the positions corresponding to $\Gamma$ and $1_S$ in the remaining positions. 
		
		Then, in the first case, $\theta$ is an $\Omega$-invariant character of $N$ extending rationally to $\psi \in \Irr(G_\theta)$. In the second case, by Lemma \ref{liftingTheta}, there exists an irreducible character $\psi \in \Irr_{p', \Omega}(B_0(G_\theta){\mid}\theta)$.
		
		Notice that, if $g \in G_\theta$, then $g$ stabilizes $\Gamma$ and hence $gH \in \Stab_{\overline{G}}(\Gamma)$. Thus, $G_\theta H/H\subseteq {\rm Stab}_{\overline{G}}(\Gamma)$ and therefore $p \mid |G : G_\theta|$. Then, in both cases, $\psi^{G}$ is irreducible, $\Omega$-invariant, and has degree divisible by $p$. In the second case, by \cite[Corollary 6.2]{NavarroBlocks} and Brauer's Third Main Theorem, it also lies in the principal block of $G$.
		
		Now, we may assume the action is $p$-concealed. Take the \textit{blocks of maximal order} $\Delta_i$ that partition $\Lambda$, in the sense of \cite[Lemma 2.4]{MinimalHeights}. Then, by that lemma, there exists $H \leq L \lhd G$ such that $G/L$ acts primitively on $\{\Delta_1, \ldots, \Delta_r\}$, and this action is $p$-concealed. Also, as $\bO^{p'}(G) = G$, $p$ divides $G/L$. By Lemma \ref{pconomega}, there exists an $\Omega$-invariant character of $G$ with degree divisible by $p$, which finishes off the first case.
		
		For the second one, let $Q \in \Syl_p(S)$ and let $R = Q \times \cdots \times Q \in \Syl_p(N)$. Let $g \in G \setminus H$. Then, $g$ acts non-trivially on one of the $S_i$, say $S_1$. Then, if $1 \neq x \in Q$, $g\cdot (x, 1, \ldots, 1) \neq (x, 1, \ldots, 1)$, and so, $g \not \in \bC_G(R)$. Consequently, $\bC_G(R) \subseteq H$. Let $U \in \Syl_p(L)$ contain $R$ (notice how $N \subseteq H \subseteq L$). Then, $\bC_G(U) \subseteq \bC_G(R) \subseteq H \subseteq L$. By \cite[Lemma 4.2]{MinimalHeights}, $\Irr(G/L) \subseteq \Irr(B_0(G))$. Finally, by Lemma \ref{pconomega}, there exists $\chi \in \Irr_\Omega(G/L)$ with $p \mid \chi(1)$, as desired.
	\end{proof}
	
	\section{Proof of Theorem A}
	We are now ready to prove Theorem A, which we restate for convenience.
	
	\begin{thmA}
		Let $G$ be a finite group, let $p$ be a prime number and let $P\in{\rm Syl}_p(G)$. Suppose that the elements in ${\rm Irr}_{\Omega}(B_0(G))$ are all of $p'$-degree. Then $P$ is abelian.
	\end{thmA}
	
	\begin{proof}
		We break the proof down into a series of steps.
		
		\bigskip
		
		\textit{Step $0$. We may assume $\bO^{p'}(G)=G$ and $\bO_{p'}(G)=1$. Moreover,
			if $1<N\lhd G$, $G/N$ has abelian Sylow $p$-subgroups. Also, $G$ has a unique minimal normal subgroup $N$ and $p\neq 2$.}
		\medskip
		
		First, suppose $\bO^{p'}(G) < G$. If there exists $\theta \in \Irr_{\Omega}(B_0(\bO^{p'}(G)))$ of degree divisible by $p$, then, by Lemma \ref{3.2BHZG}, we could find $\chi \in \Irr_{\Omega}(B_0(G)\mid\theta)$. Since $\theta(1)$ divides $\chi(1)$, $\chi(1)_p > 1$, which goes against our hypothesis. Thus, $\bO^{p'}(G)$ also fulfills our hypothesis, meaning it has abelian Sylow $p$-subgroups by induction. As $[G:\bO^{p'}(G)]$ is not divisible by $p$, this implies $G$ has abelian Sylow $p$-subgroups and we are done.
		
		Now, let $1 \neq N \lhd G$. Then, $\Irr_{\Omega}(B_0(G/N)) \subseteq \Irr_{\Omega}(B_0(G))$, meaning $G/N$ also satisfies the hypothesis. By induction, it has abelian Sylow $p$-subgroups. In particular, assume $\bO_{p'}(G) > 1$. Then, $G/\bO_{p'}(G)$ has abelian Sylow $p$-subgroups and if $P \in \Syl_p(G)$, then $P \cong P\bO_{p'}(G)/\bO_{p'}(G) \in \Syl_p(G/\bO_{p'}(G))$, meaning $P$ is also abelian, as we wanted to prove.
		
		Finally, let $M, N$ be distinct minimal normal subgroups of $G$. Then, the diagonal map composed with the canonical projections gives an embedding of $G$ into $G/N \times G/M$. By our previous paragraph, both of these have abelian Sylow $p$-subgroups. Since the Sylow $p$-subgroups of $G$ embed into those of $G/N \times G/M$, they are abelian as well. The fact that we can assume $p \neq 2$ follows from \cite[Theorem 2]{MMRSF}.
		
		\bigskip
		
		\textit{Step 1. Let $N$ be the unique minimal normal subgroup of $G$. We may assume that $N$ is not semisimple; in particular, $\bO_p(G) > 1$.}
		\medskip
		
		Suppose $N = S_1 \times \cdots \times S_t$ is a product of isomorphic non-abelian simple groups of order divisible by $p$, which are transitively permuted by $G$-conjugation. If $t > 1$, then the result follows by Step 0 and Lemma \ref{SemisimpleCase}. Now, if $t = 1$, by the uniqueness of $N = S$, $\bC_G(S) = 1$ and $G$ is an almost simple group with socle $S$, without proper normal subgroups of $p'$-index. Then, the result holds by \cite[Theorem 2.17]{MMRSF}.
		
		\bigskip
		
		\textit{Step 2. If $N\subseteq\bZ(G)$, then we are done.}
		
		\medskip

		Suppose that $N\subseteq \bZ(G)$ (so $|N|=p$). Then, $\bO_{p'}(G/N) = N/N$. Otherwise, if $N \subsetneqq K$ is such that $K/N = \bO_{p'}(G/N)$, then $K = XN$, by the Schur--Zassenhaus Theorem. But the centrality of $N$ implies $X \unlhd K$, and thus $X$ is characteristic in $K$, hence normal in $G$, contradicting $N$ being the unique minimal normal subgroup of $G$. 
		
		Since $G/N$ has abelian Sylow $p$-subgroups, by \cite[Theorem 4.1]{MinimalHeights}, we have that $G$ is the central product of $X$ and $S_1,\ldots, S_t$, where $Y/N\cong S_1/N\times\cdots\times S_t/N$. Since $N$ is the unique minimal normal subgroup of $G$, $N\subseteq S_i'$. Then, as $S_i/N$ is non-abelian simple, we have that $S_i$ is perfect, hence $S_i$ is a quasi-simple group with center $N$, for every $i$. 
		
		Let $1_N\neq\lambda\in\Irr(N)$. By \cite[Theorem 2.9]{MMRSF} there exists $\psi_i \in \Irr_\Omega(B_0(S_i))$ of degree divisible by $p$ lying over $\lambda$ (if necessary, replacing $\psi_i$ by a Galois conjugate). Now, letting $\xi\in{\rm Irr}(X|\lambda)$, then $\xi\in{\rm Irr}_\Omega(B_0(X)|\lambda)$. By \cite[Lemma 4.1]{MN} we have that the central product of characters 
		$$\chi=\xi\star  \psi_1\star\psi_2\star \cdots\star \psi_t$$ lies in the principal block of $G$. Hence $\chi\in{\rm Irr}_\Omega(B_0(G))$ has degree divisible by $p$, a contradiction.
		\bigskip

		\textit{Step 3. We may assume that $\bF(G)=\bF^*(G)$.}
		\medskip

		Since $\bO_{p'}(G)=1$ we have that $F=\bF(G)=\bO_p(G)>1$. Suppose that $E=\bE(G)>1$ and let $Z=\bZ(E)$.
		Since $N$ is the unique minimal normal subgroup of $G$, $N\subseteq Z$ (notice
		that $Z>1$ since otherwise $\bF^*(G)=\bF(G)\times E$ in contradiction with the fact that $N\subseteq E$). We claim that $E/Z=S_1/Z\times\cdots\times S_n/Z$, where
		$S_i\trianglelefteq G$ for every $i$. Let $W/Z$ be a (non-abelian) minimal
		normal subgroup of $G/Z$ contained in $E/Z$. By the Schur--Zassenhaus Theorem
		and Step~0, we know that $|W/Z|$ is divisible by $p$. Now, since $G/Z$ has abelian Sylow subgroups, we have that $W/Z$ is simple and the claim
		follows.
		Write $S=S_1$, so that $S'$ is a quasi-simple normal subgroup of $G$. Using
		again that $N$ is the unique minimal normal subgroup of $G$, we have that
		$N\subseteq \bZ(S)\cap S'\subseteq \bZ(S')$.
		Looking at the Schur multipliers of the simple groups \cite{GLS}, if $p\geq 5$,
		we deduce that $\bZ(S')$ has cyclic Sylow $p$-subgroups. We claim that
		$\bZ(S')$ has cyclic Sylow $p$-subgroups for $p=3$ 
		as well.
		
		Indeed, it can be checked in \cite{gap} that the unique simple group $S$ whose Schur
		multiplier has a non-cyclic Sylow $3$-subgroup is $\PSU_4(3)$, but this group
		does not have abelian Sylow $3$-subgroups.
		
		In all cases, we conclude that $N$ is cyclic and hence $|N|=p$. Now, the order
		of $G/\bC_G(N)$ divides $p-1$. Using Step 0 we conclude that $N$ is central
		in~$G$ and we are done by Step 2. Therefore, we may assume that $E=1$, so that 
		$\bF(G)=\bF^*(G)$, as desired.
		\bigskip
		
		\textit{Step 4. We may assume that $G$ has a unique $p$-block}.
		
		\medskip
		
		Since $\bO_{p'}(G)=1$ and we have that $\bO_p(G)=\bF(G)=\bF^*(G)$ by Step 3, then $\cent G {\bO_p(G)}\subseteq\bO_p(G)$, by Hall--Higman's Lemma 1.2.3 (\cite[Theorem 3.21]{FGT}, for instance). In this situation, $G$ has a unique $p$-block, since all blocks of $G$ cover the unique block of $\bO_p(G)$, but by \cite[Corollary 9.21]{NavarroBlocks}, the principal block of $G$ is the only block covering it.
		
		\bigskip
		
		\textit{Step 5.} We may assume that if $\lambda\in{\rm Irr}(N)$, then $p$ does not divide $|G:G_\lambda|$.
		
		\medskip
		
		We first prove that there exists an $\Omega$-invariant $\chi\in{\rm Irr}(G_\lambda)$. Let $R\in{\rm Syl}_p(G_\lambda)$ and let $\psi\in{\rm Irr}(R|\lambda)$ of minimum degree among the characters in ${\rm Irr}(R|\lambda)$. Let $\Delta=\psi^{G_\lambda}$, so $\Delta$ is an $\Omega$-invariant character (notice that $\psi$ is $\Omega$ invariant because $R$ is a $p$-group). Now notice that if $\chi\in{\rm Irr}(\Delta)$, then $\chi$ lies over $\psi$ and hence over $\lambda$. Since $\lambda$ is $G_\lambda$-invariant, if $\xi\in{\rm Irr}(R)$ is a constituent of $\chi_R$, then $\xi$ lies necessarily over $\lambda$ and then $\xi(1)\geq \psi(1)$, and since they are $p$-powers, $\psi(1)\mid\xi(1)$. Hence, if $\xi$ is an irreducible constituent of $\Delta_R$, we have that $\psi(1)\mid\xi(1)$. Write $\Delta_R=a_1\xi_1+\ldots + a_k\xi_k$, with $\xi_i\in{\rm Irr}(R)$. 
		
		Then
		$$|G_\lambda:R|\psi(1)=\Delta(1)=a_1\xi_1(1)+\ldots +a_k\xi_k(1)=\psi(1)\left(a_1\frac{\xi_1(1)}{\psi(1)}+\ldots + a_k\frac{\xi_k(1)}{\psi(1)}\right),$$ and hence 
		$$|G_\lambda:R|=a_1\frac{\xi_1(1)}{\psi(1)}+\ldots + a_k\frac{\xi_k(1)}{\psi(1)},$$ and, since $|G_\lambda:R|$ is a $p'$-number, we conclude that there exists $i$ such that $a_i$ is not divisible by $p$. Now by \cite[Lemma 2.1 (i)]{NT} applied to $G_\lambda$, $R$, $A=\Omega$ and $\Delta$, we conclude that there exists an $\Omega$-invariant $\tau\in{\rm Irr}(G_\lambda)$ with $[\Delta,\tau]\neq 0$. In particular, $\tau$ lies over $\lambda$ and is $\Omega$-invariant. Then $\tau^G\in{\rm Irr}(G)$ is $\Omega$-invariant and by our hypothesis, combined with Step $4$, $p\nmid\tau^G(1)$. Then $|G:G_\lambda|$ is a $p'$-number, as wanted.
		\bigskip

		\textit{Step 6. We may assume that $\bO_{p'}(G/N)>1$.}
		\medskip
		
		Recall that by Step~0, $G/N$ has abelian Sylow $p$-subgroups. Suppose
		$\bO_{p'}(G/N)=1$. Then, since $G/N$ has abelian Sylow $p$-subgroups, using \cite[Theorem 4.1]{MinimalHeights},
		there exist $X,Y\trianglelefteq G$ containing $N$ such that $X/N$ is an abelian
		$p$-group, $Y/N$ is a direct product of non-abelian simple groups $V_i/N$ of
		order divisible by $p$ and $G/N=X/N\times Y/N$. Furthermore $V_i/N$ is normal
		in $G/N$, since it is normal in $Y/N$ and commutes with all elements of $X/N$. Notice that $X$ is a $p$-group, so $1<\bZ(X)\lhd G$. Therefore,
		$N\subseteq\bZ(X)$. Then $X\subseteq \bC_G(N)$. Moreover, since $V_i/N$ is
		simple, $\bC_{V_i}(N)=V_i$ or $\bC_{V_i}(N)=N$. If the former happens for
		some~$i$, then $N\subseteq \bZ(V_i)$ and hence, as in Step~3, it is cyclic.
		Then $|N|=p$, and $|G:\cent G N|$ divides $p-1$. By Step 0 we conclude that $N\subseteq {\rm\textbf{Z}}(G)$ and hence we are done by Step 2. Therefore $\bC_{G}(N)=X$.
		\medskip
		
		Using Step 5, we now apply Theorem~\ref{t1} and Theorem~\ref{t3} to the action
		of $G/X\cong Y/N$ on $\Irr(N)$. Suppose first that this action is primitive, so
		that Theorem \ref{t1} applies. 
		Since $G/X$ is not solvable, we are not in case (2). Suppose now that we are in
		case~(3). In subcases~(i) and (v), $p=2$ in contradiction with Step 0. Since $G/X\cong Y/N$ is a direct product of simple groups, subcase~(ii) does
		not occur. In subcases~(iii) and (iv), we may assume that $G/X=\PSL_2(11)$ or $G/X=M_{11}$ and that $p=3$, respectively.
		In subcase (iii), there exists $\chi\in\Irr_\Omega(G/X)$ such that $\chi(1)=12$, a contradiction. Suppose that $G/X\cong Y/N=M_{11}$ and $p=3$. Then there is a rational character of degree 45 in ${\rm Irr}(G/X)$, another contradiction.
		\medskip
		
		Finally, assuming we are in case (1), we may argue as in Step 8 of the proof of \cite[Theorem 4.6]{MinimalHeights}. Doing so, we see that all of the cases we need to consider -- namely, the simple groups with abelian Sylow $p$-subgroup that appear in \cite[Appendix~1]{liebeck} -- require $p = 2$, and this goes against Step 0.
		\medskip
		
		Now, we may assume that the action of $L=Y/N$ on $N$ is imprimitive. 
		We apply Theorem \ref{t3} (recall that $L=\bO^{p'}(L)$ by Step 0). Let
		$N=N_1\oplus\cdots\oplus N_n$ be an imprimitivity decomposition for $N$.
		Therefore, $\bN_L(N_i)$ is transitive on $N_i\setminus\{0\}$ and
		$M:=L/\bigcap\bN_L(N_i)$ induces a primitive $p$-concealed subgroup of $\fS_n$.
		Note that $M$ is a factor group of $G$. By Lemma~\ref{pconomega}, this group has an $\Omega$-invariant
		irreducible character $\chi$ such that $\chi(1)_p>1$ and the result follows in
		this case.
		\bigskip
		
		\textit{Step 7. We may assume that $N=\bF(G)$; in particular, $\bC_G(N)=N$.}
		\medskip
		
		Let $K/N=\bO_{p'}(G/N)$. By Step 6, we may assume that $K>N$. By the
		Schur--Zassenhaus Theorem, there exists a $p$-complement $H$ of $K$, so $K=HN$
		and $H\cap N=1$. Then by the Frattini argument we have $G=N\bN_G(H)$.
		Write $L=\bN_G(H)$. Now, since $N$ is abelian and normal in $G$, $\bN_N(H)$ is also
		normal in $G=N\bN_G(H)=NL$, and so either $\bN_N(H) =1$ or $\bN_N(H) = N$. If
		$\bN_N(H)=N$, then $H\lhd G$, and we get a contradiction since $\bO_{p'}(G)=1$
		by Step~0. Thus, $L\cap N=\bN_N(H) =1$ and $L$ is a complement of $N$ in~$G$.
		
		Let $F=\bF(G)=\bO_p(G)$. We claim that $F=N$. Notice that $N\subseteq\bZ(F)$
		since $\bZ(F)>1$. Let $F_1=F\cap L$. Then $F_1\lhd L$ and since $G=NL$, we have
		that $F_1\lhd G$. Since $N$ is the unique minimal normal subgroup of $G$, this forces
		$F_1=1$, so $F=N$ as claimed. Also, since $N=\bF(G)=\bF^*(G)$ by Step 3, we have
		$\bC_G(N)\subseteq N$, as wanted.
		\bigskip
		
		\textit{Step 8. Completion of the proof.}
		\medskip

		Let $K/N=\bO_{p'}(G/N)$. By Step~6, we may assume that $K>N$. Again, by Step~0 we have that there exist $X,Y\trianglelefteq G$
		containing $N$ such that $X/K$ is an abelian $p$-group, $Y/K$ is a direct
		product of non-abelian simple groups $V_i/K\trianglelefteq G/K$ of order
		divisible by $p$ and $G/K=X/K\times Y/K$.
		
		As in Step~7, let $H$ be a $p$-complement of $K$ and let $L=\bN_G(H)$. Then
		$G=NL$ and $N\cap L=1$. By Step~7, we may assume that $L$ acts faithfully and
		irreducibly on $\Irr(N)$. By Step~5, this action can be assumed to be $p$-exceptional.
		\medskip
		
		Suppose first that the action of $L$ on $\Irr(N)$ is primitive. We apply
		Theorem~\ref{t1}. In case (2), $G$ is solvable and we are done by \cite[Theorem 3]{MMRSF}.
		Assume that we are in case (3). In subcases (i) and (v) have that $p=2$, in contradiction with Step 0. In subcases (ii), (iii) and (iv) $G/N$ has a normal subgroup
		$V/N=\SL_2(5), \PSL_2(11)$ or $M_{11}$ respectively and in all cases $p=3$.
		Note that one of the simple direct factors of $Y/K$ is then $\fA_5,\PSL_2(11)$
		or $M_{11}$, respectively. Since they all have $\Omega$-invariant irreducible characters of degree divisible by $3$, the result follows in these cases.
		\medskip 
		
		Now, we may assume that we are in case (1). So $L$ is transitive on
		$\Irr(N)\setminus\{1_N\}$ and $L$ is one of the groups from Hering's theorem.
		Again, we use the description in \cite[Appendix~1]{liebeck}. 
		We start with the infinite classes described in (A). Since $G/N$ is not
		solvable and has abelian Sylow $p$-subgroups, it follows that $G/N$ has a
		normal subgroup $V/N\cong \SL_2(p^n)$, where $|N|=p^{2n}>3$. Now, write $G/K=X/K\times Y/K=X/K\times Y_1/K\times\cdots\times Y_t/K$. Arguing as in \cite{MinimalHeights} we have that there is some $i$ such that $Y_i/K\cong \PSL_2(p^n)$. Write $M/K=Y_i/K$ and let $\varphi\in{\rm Irr}_\Omega(M/K)$ of degree divisible by $p$ (using \cite[Theorem 2.8]{MMRSF}, for instance). Then $\chi=1\times \cdots\times 1\times \varphi\times 1\times\cdots\times 1\in{\rm Irr}(G/K)\subseteq{\rm Irr}(G)$ is $\Omega$-invariant and of degree divisible by $p$.
		\medskip
		
		Next, we consider the extra-special classes described in (B). The first four
		groups in \cite[Table~10]{liebeck} are solvable, so they do not occur. The last case
		of this table (where $|N|=3^4)$ can be handled with \cite{gap}. In this case we have that $L$ ($G_0$ in the notation of Liebeck) has a normal $2$-subgroup $R$ such that $L/R$ is isomorphic to a subgroup of $\fS_5$. Since $G$ is not solvable, $L/R$ cannot be solvable and hence $L/R\cong\fA_5$ or $L/R\cong \fS_5$. As before, this means that one of the simple direct factors of $Y/K$ is isomorphic to $A_5$ and we know it possesses an $\Omega$-invariant character of degree divisible by $3$.
		\medskip
		
		Finally, we consider the exceptional classes described in (C). In the cases
		where $p^d\in\{11^2,19^2,29^2,59^2\}$ in \cite[Table~11]{liebeck}, we have that
		$G/N$ (and hence, $G$) is $p$-solvable, so we are done by \cite[Theorem 3]{MMRSF}. Since $G/N$ has abelian
		Sylow $p$-subgroups, we are left with the first and the last cases of Table~11.
		In the last case, $p=3$ and $G/N\cong\SL_2(13)$. Here, there is
		$\chi\in\Irr(G/N)$ of degree~6 which is $\Omega$-invariant. The first case can be handled as above, since we again have $\fA_5$ as a simple direct factor of $Y/K$.
		\medskip
		
		Now, assume that the action of $L$ on $N$ is imprimitive. Arguing as in the
		last paragraph of Step 6, we find $\chi\in\Irr_\Omega(L)$ such that $\chi(1)_p>1$. This final contradiction completes the proof.
		
	\end{proof}
	
The following will be used in the proof of Theorem B. 	
	
	\begin{cor}
		\label{GritV1}
		Let $G$ be a finite group, with $p$ a prime and suppose that $\Irr_{\Omega}(B_0(G)) \subseteq \Irr_{p'}(G)$. Then, if $N = \bO_{p'}(\bO^{p'}(G))$, $\bO^{p'}(G)/N \cong X/N \times Y/N$, where $X/N$ is an abelian $p$-group and $Y/N$ is either trivial or a direct product of non-abelian finite simple groups of order divisible by $p$ without $\Omega$-invariant characters of degree divisible by $p$ in the principal block.
	\end{cor}
	
	\begin{proof}
		By Theorem A, $G$ has abelian Sylow $p$-subgroups. The result then follows by \cite[Theorem 4.1]{MinimalHeights}. The fact that the simple direct factors of $Y/N$ do not possess any $\Omega$-invariant characters of degree divisible by $p$ in the principal block follows from Lemma \ref{3.2BHZG} and the fact that $Y/N \cong \bO^{p'}(G)/X$.
	\end{proof}
	
	\medskip
	
	\section{Proof of Theorem B}
	
	We may now use the results of the previous section to obtain a stronger version of the main theorem of \cite{Grit}. We remark that there seems to be an unfortunate mistake in the 7th paragraph of the proof of that result. There, the author claims that, embedding a group $G$ into a group $\Gamma$, such that $\bO_p(G) \subseteq \bO_p(\Gamma)$ and the latter has a normal complement in $\Gamma$, then there is a normal complement of $\bO_p(G)$ in $G$. A counterexample is given by $\Gamma = C_4 \times S_3$ and $G = C_3 \rtimes C_4$ for $p = 2$.
	
	In order to do this, we need to cite some slight modifications of results of \cite{GritPSolvable}. As the proofs are very similar to those of \cite{GritPSolvable}, we omit them.
	
	\begin{lem}
		\label{GritLemma1}
		Let $q \neq p$ be a prime, let $P$ be a $p$-group and let $V$ be an irreducible, finite dimensional $KP$-module, where $K$ is a finite splitting field for $P$ of characteristic $q \neq p$, such that $\bC_V(P) = 0$. Let $\Omega$ be a set of automorphisms of $V$ of order $p$ and which commute with the action of $P$. Then, for each $\sigma \in \Omega$, there exists $g_\sigma \in P$ such that $\sigma v = g_\sigma v$, for all $v \in V$.
	\end{lem}
	
	\begin{lem}
		\label{GritLemma2}
		Let $q \neq p$ be a prime, let $P$ be a $p$-group and let $V$ be a finite dimensional $\FF_qP$-module, and assume $\bC_V(P) = 0$. Let $\Omega$ be a set of automorphisms of $V$ of order $p$ and which commute with the action of $P$. Assume also that $\sigma v \in \operatorname{span}_{\FF_q} \langle v \rangle$, for all $v \in V, \sigma \in \Omega$. Then, there exists $0 \neq v \in V$ such that $|P:\bC_P(v)| = \min\{|P:\bC_P(w)| \mid 0 \neq w \in V\}$ and, for each $\sigma \in \Omega$, there exists some $g_\sigma \in P$ with $\sigma v = g_\sigma v$.
	\end{lem}
	
	The following is a version of \cite[Theorem B]{GritPSolvable} adapted to use $\Omega$ instead of an automorphism $\sigma \in \Omega$. The proof is not fundamentally different from the one in \cite{GritPSolvable}, and so we omit it.
	
	\begin{teo}
		\label{GritThmB}
		Let $P \in \Syl_p(G)$ and suppose $\Irr_\Omega(G) \subseteq \Irr_{p'}(G)$. Then, $\bN_G(P)$ intersects every minimal normal subgroup of $G$ non-trivially.
	\end{teo}
	
	We are now ready to prove Theorem B. We do this by breaking it down into three separate parts, which we prove individually.
	
	\begin{lem}
		\label{Part1}
		Let $G$ be a finite group with $p$ a prime and suppose that $\Irr_{\Omega}(G) \subseteq \Irr_{p'}(G)$. Then, $\bO^{p'}(G) = \bO_p(G) \times K$, where $K \unlhd G$ is such that $K/\bO_{p'}(K)$ is a (possibly empty) direct product of non-abelian finite simple groups of order divisible by $p$.
	\end{lem}

	\begin{proof}
		Let $G$ be a minimal counterexample to the assertion in the theorem. Write, as in Corollary \ref{GritV1}, $\bO^{p'}(G)/N = X/N \times Y/N$, where $N = \bO_{p'}(\bO^{p'}(G))$. Notice that $X = HN$, where $H$ is an abelian $p$-group; from this, it follows that $\bO^{p'}(G) = HY$. Notice also that $\bO_p(G)\subseteq\bO^{p'}(G)$ and hence $\bO_p(G)N/N\subseteq X/N=\bO_p(\bO^{p'}(G)/N)$. It follows that $\bO_p(G) \subseteq X$; in particular, it is contained in $H$.
		\bigskip
	
		\textit{Step 1. $\bO^{p'}(G) = G$.}
		\medskip
	
		Assume otherwise. By Lemma \ref{3.2BHZG}, $\bO^{p'}(G)$ also satisfies our hypothesis, since each of its $\Omega$-invariant irreducible characters lies under some $\Omega$-invariant irreducible character of $G$. Then, by the minimality of $G$, we have $\bO^{p'}(G) = \bO_p(\bO^{p'}(G)) \times K$, for some $K \unlhd \bO^{p'}(G)$ in the conditions of the statement of the lemma, as $\bO^{p'}(\bO^{p'}(G)) = \bO^{p'}(G)$. But also $\bO_p(G) \subseteq \bO^{p'}(G)$. Thus $\bO^{p'}(G) = \bO_p(G) \times K$. Finally, since $K/\bO_{p'}(K)$ is a (possibly empty) direct product of non-abelian finite simple groups of order divisible by $p$, $\bO^p(K) = K$ and thus $K = \bO^p(\bO^{p'}(G))$, meaning $K$ is normal in $G$. This violates $G$ being a counterexample.
		\bigskip
	
		\textit{Step 2. $\bO_p(G) = 1$.}
		\medskip
	
		If $M=\bO_p(G)>1$, write $\bar{G}=G/M$. Since $\bar{G}$ satisfies our hypothesis, we have, by the minimality of $G$, that $\bO^{p'}(\bar{G})=\bar{K}$, where $\bar{K}/\bO_{p'}(\bar{K})$ is a (possibly empty) direct product of non-abelian simple groups of order divisible by $p$. Notice that $\bO^{p'}(\bar{G})=\bO^{p'}(G)/M = G/M$. Since $X/N = \bO_p(G/N)$ and $MN/N$ is a normal $p$-subgroup of $G$, $MN \subseteq X$. But also $|X:MN|$ is a power of $p$ and $G/X$ is a direct product of non-abelian simple groups of order divisible by $p$, meaning $G/NM$ has no proper normal subgroups of $p'$-order. Consequently, $MN/M = \bO_{p'}(G/M)$. We conclude that $G/MN$ is the direct product of non-abelian simple groups and therefore, $X=MN$ which implies that $H=M$, from which $G = M \times Y$, which violates it being a counterexample.
		\bigskip
	
		\textit{Step 3. If $M \subseteq N$ is a minimal normal subgroup of $G$, then $M$ is not abelian.}
		\medskip
	
		First, notice that if $N=1$ we have the desired result, so we can exclude this case. Now, suppose that $M$ is abelian. Let $P$ be a Sylow $p$-subgroup of $G$ containing $H$. Then, by Theorem \ref{GritThmB}, we have that 
		\begin{equation*}
			1<\bN_G(P)\cap M=\bC_M(P)\leq \bC_M(HM)=M\cap \bZ(HM). 
		\end{equation*}
		We claim that $\bZ(HM)$ is normal in $G$. Indeed, by induction $G/M=\bO^{p'}(G/M)=\bO_p(G/M)\times K_1/M$. Notice that $\bO_p(G/M)\subseteq X/M$ and therefore $\bO_p(G/M)\subseteq HM/M$. 
	
		Notice that $N \subseteq K_1$, and hence $\bO_{p'}(K_1) = N$. Then $K_1/N$ is a (possibly empty) direct product of non-abelian simple groups of order divisible by $p$. Also notice that $K_1 = \bO^p(G)$, as $K_1 = \bO^p(K_1)$ and $|G:K_1|$ is a power of $p$. Since the same is true of $Y$, it follows that $Y = K_1$.
	
		Therefore, we conclude that $HM/M=\bO_p(G/M)$, by comparing orders. Hence $HM$ is normal in $G$, and so is $\bZ(HM)$, as wanted. Since $M\cap \bZ(HM)>1$, we conclude that $M\subseteq\bZ(HM)$. Hence $H$ is normal in $HM$ and therefore characteristic, so it is normal in $G$. As $\bO_p(G) = 1$, $H = 1$ and $G = Y$, a contradiction to $G$ being a counterexample.
		\bigskip
	
		\textit{Step 4. Final contradiction.}
		\medskip
	
		By Steps 2 and 3, $M$ is the product of non-abelian simple groups of order not divisible by $p$, transitively permuted by $G$-conjugation. Write $M = S_1 \times \cdots \times S_k$. By Lemma \ref{SemisimpleCase}, $k = 1$, meaning $M = S$, a non-abelian simple group of $p'$-order. In particular, $G/\bC_G(S)$ is an almost simple group with socle $S\bC_G(S)/\bC_G(S)\cong S$. 
	
		Now, $G/S\bC_G(S)\cong \frac{G/\bC_G(S)}{S\bC_G(S)/\bC_G(S)}$ is solvable by Schreier's conjecture. Write $W/S=\bO_p(G/S)$, and notice that, by induction $G/S=W/S\times K_1/S$ for some $K_1$ such that $\frac{K_1/S}{\bO_{p'}(K_1/S)}$ is a product of non-abelian simple groups. Let $L=W\bC_G(S)$. Then, $G/L$ is solvable and hence $K_1/(K_1\cap L)$ is solvable. Thus, $K_1=(K_1\cap L)N$ and we have that $K_1/(K_1\cap L)$ is a $p'$-group. But since $G=WK_1$, we have that $G=LK_1$ and hence $G/L\cong K_1/(K_1\cap L)$ is a $p'$-group. As $\bO^{p'}(G)=G$, we obtain that $G=L$ and hence $G/S\bC_G(S)\cong W/(W\cap S\bC_G(S))$ is a $p$-group.
	
		By the proof of \cite[Theorem 2.8]{MMRSF}, if $G \neq S\bC_G(S)$, since $p$ does not divide $|S|$, there exists $\chi \in \Irr_{\Omega}(G/\bC_G(S))$ of degree divisible by $p$, which goes against our hypothesis. Thus, $G = S\bC_G(S)$. But this is a contradiction, as $S \cap \bC_G(S) = 1$ and $|S|$ is not divisible by $p$.
	\end{proof}

	\begin{lem}
		\label{Part2}
		Let $G$ be a finite group with $p$ a prime and suppose that $\Irr_{\Omega}(G) \subseteq \Irr_{p'}(G)$. Then, $\bO^{p'}(G) = \bO_p(G) \times K$, where $K \unlhd G$ is such that $K/\bO_{p'}(K)$ is a (possibly empty) direct product of non-abelian finite simple groups of order divisible by $p$ without $\Omega$-invariant characters of degree divisible by $p$.
	\end{lem}

	\begin{proof}
		By Lemma \ref{Part1}, we may write $\bO^{p'}(G) = \bO_p(G) \times K$ where $K/\bO_{p'}(K)$ is a direct product of non-abelian finite simple groups of order divisible by $p$. If these groups have an $\Omega$-invariant character of degree divisible by $p$, then so does $K/\bO_{p'}(K)$, and hence, so does $\bO^{p'}(G)/\bO_p(G)$. This means that $\bO^{p'}(G)$ has an $\Omega$-invariant character of degree divisible by $p$ and, by Lemma \ref{3.2BHZG}, so does $G$, a contradiction.
	\end{proof}
	
	Finally, we get the complete proof of Theorem B, which we restate below.
	
	\begin{thmB}
		Let $G$ be a finite group, with $p$ a prime and suppose that $\Irr_{\Omega}(G) \subseteq \Irr_{p'}(G)$.
		Then, $\bO^{p'}(G) = \bO_p(G) \times K$, where $K \unlhd G$, $\bO_{p'}(K)$ is solvable and $K/\bO_{p'}(K)$ is a (possibly empty) direct product of non-abelian simple groups of order divisible by $p$ without $\Omega$-invariant characters of degree divisible by $p$.
	\end{thmB}
	
	\begin{proof}
		By what we did before in Lemma \ref{Part1} and Lemma \ref{Part2}, $\bO^{p'}(G) = \bO_p(G) \times Y$, $Y/N$ is a direct product of non-abelian finite simple groups of order divisible by $p$ without $\Omega$-invariant characters of degree divisible by $p$, where $N = \bO_{p'}(\bO^{p'}(G))$, and $\bO^{p'}(G) = XY$, where $X = \bO_p(G)N$. Also, $Y/N = \bO^p(\bO^{p'}(G)/N)$, meaning $Y/N$ is normal in $G/N$; hence, $Y$ is normal in $G$.
		
		All that is left is to show that $\bO_{p'}(Y)$ is solvable. In view of that, we may assume $p \neq 2$ by the Feit--Thompson theorem. If $\bO^{p'}(G) \neq G$, then, by Lemma \ref{3.2BHZG}, we get the desired result by induction. Hence, we may assume $G = \bO^{p'}(G)$. Along the same lines, as $G/\bO_p(G) \cong Y$, if $Y$ had an $\Omega$-invariant character of degree divisible by $p$, so would $G$. Thus, $Y$ also satisfies our hypothesis.
		
		Notice that $\bO^{p'}(Y) = Y$, since $Y \cong G/\bO_p(G)$ and $\bO^{p'}(G) = G$. Consequently, if $Y \neq G$, by induction, we get that $\bO_{p'}(Y)$ is solvable, which is our desired result. So we may assume $Y = G$. In this situation, $\bO_p(G) = 1$, $N = \bO_{p'}(G)$ and $G/N$ is a direct product of non-abelian simple groups of order divisible by $p$ whose $\Omega$-invariant characters are of $p'$-degree. Let $R = \mathbf{R}(N)$ be the solvable radical of $N$ and suppose $R \neq 1$. Then, by induction on $G/R$, we get that $\bO_{p'}(G/R) = N/R$ is solvable, meaning so, too, is $N$. Thus, we may assume $R = 1$.
		
		Let $M$ be a minimal normal subgroup of $G$ contained in $N$. By the previous paragraph, $M = S_1 \times \cdots \times S_k$, a product of non-abelian simple groups transitively permuted by $G$-conjugation and of order not divisible by $p$. If $k > 1$, then the result follows by Lemma \ref{SemisimpleCase}. Thus, we may assume $k = 1$. Then, $M = S$ is a non-abelian simple group and $G/\bC_G(S)$ is an almost simple group with socle isomorphic to $S$. By the Schreier Conjecture, $G/S\bC_G(S)$ is solvable. In particular, due to the structure of $G$, $|G:S\bC_G(S)|_p = 1$. Consequently, $G = S\bC_G(S)$ as $\bO^{p'}(G) = G$. But this is impossible, as $p \nmid |S|$ and $S$ is a quotient of $G$.
	\end{proof}
	
	Notice how the conclusion of Theorem B fails if one considers only the $\Omega$-invariant characters in the principal block of $G$. A small counterexample is $\fS_3$, whose characters in the principal $2$-block are rational and linear, and also $\bO_2(\fS_3) = 1, \bO^{2'}(\fS_3) = \fS_3$; nevertheless, it is not a $2'$-group, as would be implied by the theorem.
	
	\begin{ack}
		We would like to thank Gunter Malle and Mandi Schaeffer Fry for kindly sharing their thoughts and comments on the manuscript.
	\end{ack}


\begin{thebibliography}{99}
		
		\bibitem{McKay} {\sc M. Cabanes, B. Späth}, The McKay Conjecture on character degrees. \emph{Ann. Math.} (to appear);
		
		\bibitem{gap}
		{\sc The GAP~Group}, \emph{GAP -- Groups, Algorithms, and Programming, Version
			4.11.0}; 2020, {\sf http://www.gap-system.org};
		
		\bibitem{GLPST}
		{\sc M. Giudici, M. Liebeck, C. Praeger, J. Saxl, P.H. Tiep}, Arithmetic
		results on orbits of linear groups. \emph{Trans. Amer. Math. Soc. \bf368}
		(2016), 2415--2467;
		
		\bibitem{GLS}
		{\sc D. Gorenstein, R. Lyons, R. Solomon}, \emph{The Classification of the
			Finite Simple Groups. Number~3.} Mathematical Surveys and Monographs,
		American Mathematical Society, Providence, RI, 1998;
		
		\bibitem{Grit}
		{\sc N. Grittini}, On the degrees of irreducible characters fixed by
		some field automorphism in finite groups. \emph{Bull. London Math. Soc.} (2025), 120--136;
		
		\bibitem{GritPSolvable}
		{\sc N. Grittini}, On the degrees of irreducible characters fixed by some field automorphism in $p$-solvable groups. \emph{Proc. Amer. Math. Soc. {\bf 151}} (2023), 4143--4151;
		
		\bibitem{HSFTVV}
		{\sc N. N. Hung, A. A. Schaeffer Fry, H. P. Tong-Viet, C. Ryan Vinroot}, On the number of irreducible real-valued characters of a finite group. \emph{J. Algebra \bf555} (2020), 275--288;
		
		\bibitem{FGT}
		{\sc I. M. Isaacs}. \emph{Finite group theory}, Graduate Studies in Mathematics, 92, Amer. Math. Soc., Providence, RI, 2008;
		
		\bibitem{Landrock}
		{\sc P. Landrock}, On the number of irreducible characters in a $2$-block. \emph{J. Algebra} (1981), 426--442;
		
		\bibitem{liebeck}
		{\sc M. W. Liebeck}, The affine permutation groups of rank three. \emph{Proc.
			London Math. Soc. {\bf54}} (1987), 477--516;
		
		\bibitem{MinimalHeights} {\sc G. Malle, A. Moretó, N. Rizo}, Minimal heights and defect groups with two character degrees, \emph{Adv. Math. {\bf 441}} (2024), 22 pp.;
		
		\bibitem{MMRSF}
		{\sc G. Malle, A. Moret\'o, N. Rizo, A. A. Schaeffer Fry}, A Brauer--Galois height zero conjecture. \emph{Int. Math. Res. Not. IMRN {\bf 2025}} (2025), 20 pp.;
		
		\bibitem{MN}
		{\sc G. Malle, G. Navarro}, Height zero conjecture with Galois automorphisms. 
		\emph{J. London Math. Soc. {\bf 107}} (2023), 548–-567;
		
		\bibitem{BHZ} {\sc G. Malle, G. Navarro, A. A. Schaeffer Fry, P. H. Tiep}, Brauer’s Height Zero Conjecture. \emph{Ann. Math. {\bf200}} (2024), 557--608;
		
		\bibitem{Galois-McKay} {\sc G. Navarro}, The McKay conjecture and Galois automorphisms. \emph{Ann. Math. {\bf160}} (2004), 1129--1140;
		
		\bibitem{NavarroBlocks} {\sc G. Navarro}. \emph{Characters and blocks of finite groups}. Volume 250 of \emph{London Mathematical Society Lecture Note Series}. Cambridge University Press, Cambridge, 1998;
		
		\bibitem{NT}
		{\sc G. Navarro, P.H. Tiep}, Sylow subgroups, exponents, and character values. \emph{Trans. Amer. Math. Soc. {\bf 372}} (2019), 4263--4291;
		
		\bibitem{NT13}
		{\sc G. Navarro, P. H. Tiep}, Characters of relative $p'$-degree over
		normal subgroups. \emph{Annals of Math. (2) \bf 178} (2013), 1135--1171;
		
		\bibitem{NTV}
		{\sc G. Navarro, P. H. Tiep, C. Vallejo}, Brauer correspondent blocks with one simple module. \emph{Trans. Amer. Math. Soc. \bf{371}} (2019), 903--922;
		
		\bibitem{wal}
		{\sc J. H. Walter}, The characterization of finite groups with Abelian Sylow
		$2$-subgroups. \emph{Annals of Math. (2) \bf89} (1969), 405--514;
		
	\end{thebibliography}
\end{document}